\documentclass[11pt, a4]{amsart}
\usepackage{amsmath, amssymb, amsthm, amsfonts, amscd, braket, amsaddr}
\usepackage[shortlabels]{enumitem}
\usepackage{hyperref}
\usepackage{graphicx}
\usepackage{tikz}
\usepackage{tikz-cd}
\theoremstyle{definition}
\newtheorem{thm}{\bf Theorem}[section]
\newtheorem*{thm*}{\bf Theorem}
\newtheorem{prop}[thm]{\bf Proposition}
\newtheorem{lem}[thm]{Lemma}
\newtheorem*{lem*}{Lemma}
\newtheorem{cor}[thm]{Corollary}
\newtheorem*{cor*}{Corollary}
\newtheorem*{defn}{\bf{Definition}}

\newtheorem{example}{Example}

\theoremstyle{remark}
\newtheorem*{rem}{\bf{Remark}}
\newcommand{\AAA}{\mathbb{A}}
\newcommand{\BB}{\mathbb{B}}

\newcommand{\FF}{\mathbb{F}}
\newcommand{\GG}{\mathbb{G}}

\newcommand{\NN}{\mathbb{N}}

\newcommand{\PP}{\mathbb{P}}

\newcommand{\ZZ}{\mathbb{Z}}

\newcommand{\fraka}{\mathfrak a}
\newcommand{\frakb}{\mathfrak b}

\newcommand{\frakm}{\mathfrak m}

\newcommand{\frakp}{\mathfrak p}

\newcommand{\frakS}{\mathfrak S}

\newcommand{\calC}{\mathcal C}

\newcommand{\calO}{\mathcal O}

\newcommand{\Hom}{\mathop{\mathrm{Hom}}\nolimits}

\newcommand{\Spec}{\mathop{\mathrm{Spec}}\nolimits}

\newcommand{\ia}{\fraka}
\newcommand{\ib}{\frakb}

\newcommand{\ip}{\frakp}

\newcommand{\im}{\frakm}

\newcommand{\twovect}[2]%
{\ensuremath{\left( \begin{array}{c}
 #1 \\
 #2 
\end{array} \right) }}

\newcommand{\thrvect}[3]%
{\ensuremath{\left( \begin{array}{c}
 #1 \\
 #2 \\
 #3
\end{array} \right) }}

\newcommand{\fouvect}[4]%
{\ensuremath{\left( \begin{array}{c}
 #1 \\
 #2 \\
 #3 \\
 #4
\end{array} \right)}}

\newcommand{\thrlen}[3]%
{\ensuremath{\left|| \begin{array}{c}
 #1 \\
 #2 \\
 #3
\end{array} \right|| }}

\newcommand{\twonude}[2]%
{\ensuremath{ \begin{array}{c}
 #1 \\
 #2 
\end{array}  }}

\newcommand{\thrnude}[3]%
{\ensuremath{ \begin{array}{c}
 #1 \\
 #2 \\
 #3
\end{array} }}

\newcommand{\founude}[4]%
{\ensuremath{ \begin{array}{c}
 #1 \\
 #2 \\
 #3 \\
 #4
\end{array} }}

\newcommand{\twomat}[4]%
{\ensuremath{\left( \begin{array}{cc}
 #1 & #2\\
 #3 & #4
\end{array} \right) }}

\newcommand{\thrmat}[9]%
{\ensuremath{\left( \begin{array}{ccc}
 #1 & #2 & #3\\
 #4 & #5 & #6\\
 #7 & #8 & #9
\end{array} \right) }}

\newcommand{\twodet}[4]%
{\ensuremath{\left| \begin{array}{cc}
#1 & #2\\
#3 & #4
\end{array} \right|}}

\newcommand{\thrdet}[9]%
{\ensuremath{\left| \begin{array}{ccc}
 #1 & #2 & #3\\
 #4 & #5 & #6\\
 #7 & #8 & #9
\end{array} \right| }}

\newcommand{\pmag}{{\mathcal{P\!M}ag}}
\newcommand{\pmon}{{\mathcal{P\!M}on}}
\newcommand{\pring}{{\mathcal{P\!R}ing}}

\newcommand{\sets}{{\mathcal{S}et}}

\newcommand{\abmon}{{\mathcal{A}b\mathcal{M}on}}

\newcommand{\grp}{{\mathcal{G}rp}}
\newcommand{\alg}{{\mathcal{A}lg}}
\newcommand{\pgrp}{\mathcal{PG}rp}

\newcommand{\fun}{\FF_1}

\newcommand{\GGLL}{{\mathbb{G}\mathbb{L}}}
\newcommand{\dslash}{/\!\!/}

\newcommand{\angles}[1]{\langle {#1} \rangle}

\title{Partially additive rings \\
and group schemes over $\fun$}
\author{Shingo Okuyama}
\email{okuyama@di.kagawa-nct.ac.jp}
\address{National Institute of Technology, Kagawa Colledge}

\begin{document}

\begin{abstract}
We develop an elementary theory of partially additive rings as a foundation of ${\mathbb F}_1$-geometry. 
Our approach is so concrete that an analog of classical algebraic geometry is established very straightforwardly. 
As applications, (1) we construct a kind of group scheme ${\GGLL}_n$ whose value at a commutative ring 
$R$ is the group of $n\times n$ invertible matrices over $R$ and at ${\mathbb F}_1$ is the $n$-th symmetric group, 
and (2) we construct a projective space $\mathbb P^n$ as a kind of scheme and count the number of points of 
${\mathbb P}^n({\mathbb F}_q)$ for $q=1$ or $q=p^n$ a power of a rational prime, then we explain a reason of 
number 1 in the subscript of ${\mathbb F}_1$ even though it has two elements. 
\end{abstract}
\maketitle
\tableofcontents
\section{Introduction}

It seems that the notion of a so-called field with one element was first proposed by J. Tits\cite{tits-sur-les-analogues-algebriques}. 
There have been many attempts to define an algebraic geometry over $\fun.$

We start from a partially additive algebraic system, partial monoids. We impose
strict associativity on partial monoids, in the sense of 
G. Segal\cite{segal-configuration-spaces-and},
who used the topological version of this structure in the study of configuration spaces.
Then we define a partial ring as a partial monoid equipped with a binary commutative, associative
multiplication with unity.  Our approach is so concrete that we can establish
an analog of the classical theory of schemes based on commutative rings with unity 
and define so-called partial schemes, which are locally partial-ringed spaces that are 
isomorphic to an affine partial scheme in an open neighborhood of each point.

Among others, rather concrete constructions of $\fun$-geometry from a
(partial) algebraic systems are, Deitmar\cite{deitmar-schemes-over-f1},
Deitmar\cite{deitmar-congruence-schemes} and 
Lorscheid\cite{lorscheid-the-geometry-of-1}.
In particular, the latter two treat the partially additive algebraic system, so 
our approach resembles their approach most. Indeed the category of
partial rings is embedded in the category of blueprints.

The main outcome of our construction is that we can describe a group valued functor
$\GGLL_n$ by a partial group object in the category of partial rings,
which means that we have constructed an affine partial group partial scheme.
When this functor is applied to good partial rings, it takes values in the category of groups.
For example, if $A$ is a commutative ring with unity, then $\GGLL_n(A)$ is nothing but
the general linear group of $n\times n$ matrices, and if $A=\fun, $ then 
$\GGLL_n(\fun) =\frakS_n$ is the $n$-th symmetric group.

Another modest outcome of our approach is that we have an explanation
why the number 1 is put to the notation of our field even though it has two elements.	
Namely, we say that the number is there since
we have only one element which can be added to 1 in the field 
$\fun = \set{0,1},$ while in the usual finite field $\FF_q,$ there are $q$ of such elements.
(See Example \ref{ex:projective}.)
In this paper, $\NN$ denotes the set of non-negative integers.

{\it Acknowledgement.}
Conversations with Bastiaan Cnossen about tensor products of partial monoids
are very useful for the author. Indeed, the associative closure and the tensor product which
appear in this article are contained in his argument \cite{bastiaan-cnossen-master-thesis},
explicitly or implicitly.
Yoshifumi Tsuchimoto suggested to the author that search for a group scheme in
$\fun$-geometry is important.
Bastiaan Cnossen, 
Katsuhiko Kuribayashi, 
Makoto Masumoto, Shuhei Masumoto, 
Jun-ichi Matsuzawa,
Kazunori Nakamoto and Yasuhiro Omoda gave valuable comments on previous versions
of this article. The author is grateful to these people.

\section{Additive Part}
\subsection{Partial Magmas and Monoids}
\begin{defn}[partial magma and partial monoid]
	A (commutative and unital) {\bf partial magma} is 
	\begin{enumerate}
		\item a set $A,$ with a distinguished element $0,$
		\item a subset $A_2 $ of $A\times A,$
		\item a map $+\colon A_2 \to A,$
	\end{enumerate}
	such that
	\begin{enumerate}[label= (\alph*)]
		\item $(0,a)\in A_2, (a,0) \in A_2$ and $a+0 = a = 0+a,$ for all $a\in A,$
		\item if $(a,b) \in A_2$ then $(b,a) \in A_2$ and $a+b = b+a,$ for all $a, b \in A.$
	\end{enumerate}

	If, moreover, it satisfies the following condition, we say that it is a  (commutative and unital) {\bf partial monoid} :
	\begin{enumerate}[label= (\alph*)]
		\setcounter{enumi}{2}
		\item $(a,b), (a+b,c) \in A_2$ if and only if $(b,c), (a,b+c)\in A_2$  for all $a, b, c \in A$
		and in such a case, $(a+b)+c = a+(b+c).$
	\end{enumerate}
\end{defn}

	If, instead, a partial magma satisfies the following condition, we say that it is a  (commutative and unital) {\bf weak partial monoid} (or
	weakly associative partial magma) :
	\begin{enumerate}[label= (\alph*)]
		\setcounter{enumi}{3}
		\item If $a_1 + \dots + a_r$ can be calculated in $A$ under supplement of parenthesis in more than or equal to two ways,
		then the results are equal for all $r\in \NN$ and $a_1, \dots, a_r \in A.$
	\end{enumerate}

\begin{defn}
	Let $A$ and $B$ be partial magmas. A map $f\colon A\to B$ is called a {\bf homomorphism} if
	$f(0) = 0, (f\times f)(A_2) \subseteq B_2$ and $f(a_1 + a_2) = f(a_1) + f(a_2)$ for all $(a_1, a_2) \in A_2.$
	If $A$ and $B$ are partial monoids, a map $A\to B$ is a {\bf homomorphism} if it is a homomorphism of partial magmas.
	The category of partial magmas and partial monoids are denoted by $\pmag$ and $\pmon,$
	respectively.
\end{defn}

\begin{example}
	A partial magma of order 1 is isomorphic to $0 = \Set{0},$ which is a partial monoid.
	A partial magma of order	2 is isomorphic to one of the following :
	\begin{enumerate} 
		\item $\FF_1 = \Set{0,1}$ where $1+1$ is undefined,
		\item $\FF_2 = \Set{0,1}$ where $1+1 = 0$ and
		\item $\BB = \Set{0,1}$ where $1+1 = 1.$
	\end{enumerate}
	These are also partial monoids.
\end{example}

\begin{example}
	Any based set $(X,0)$ can be given a partial monoid structure by putting $X_2 = \Set{0} \times X \cup X\times \Set{0}.$
	A homomorphism between such partial monoids is nothing but a based map between the given based sets. 
\end{example}

\begin{example}
	Any abelian monoid $M$ can be given a partial monoid structure by putting $M_2 = M\times M.$
	A homomorphism between such partial monoids is nothing but a homomorphism between 
	the given abelian monoids. 
\end{example}

These examples show that we can embed a category of based sets $\sets_0$ and that of abelian monoids 
$\abmon$ into $\pmon.$ In this paper, 
based sets and abelian monoids (and abelian groups) are always considered as partial monoids unless otherwise specified.

It is easily shown that a homomorphism $f\colon A\to B$ is a monomorphism 
in the category of partial monoids if and only if it is an injective map of
underlying sets. When $A$ and $B$ are partial magmas, we say that $A$ is contained in $B$ 
to mean that $A$ is a partial submagma of $B.$
Remark that a homomorphism $f\colon A\to B$ is an epimorphism if it is a surjective map of
underlying sets, but the converse is false.  For example, the map $\fun \to \NN$ determined by $1\mapsto 1$
is a monomorphism and epimorphism, but not an isomorphism.

\subsection{Monoid completion}

In this section, we show that for any partial magma $A,$ there exists 
a monoid $A_{mon}$ and a homomorphism $\mu\colon A\to A_{mon}$
which is universal among the homomorphisms from $A$ to abelian monoids.

Let $A$ be a partial magma and $\NN[A]$ be the free abelian monoid
generated by the underlying set of $A.$ More precisely,
\[
	\NN[A] = \Set{ a_1 \dotplus \dots \dotplus a_r | r\in \NN, a_i \in A\,(1\leq i\leq r)},
\]
where the empty sum is the unit. 
By the injective homomorphism $A \to \NN[A]~;~a\mapsto a$ of partial magmas,
we regard $A$ as a partial submagma of $\NN[A].$
Let $\sim$ be the equivalence relation on $\NN[A]$ generated by 
$0\dotplus x\sim x$ and $(a_1+a_2) \dotplus x \sim a_1\dotplus a_2\dotplus x,$
where $x$ is any element of $\NN[A]$ and $(a_1, a_2) \in A_2.$
We put $A_{mon} = \NN[A]/\sim.$ Since $\sim$ is an additive equivalence relation, $A_{mon}$ has a
monoid structure such that the projection $\NN[A]\to A_{mon}$ is a homomorphism of monoids.
The composite $A\to \NN[A] \to A_{mon}$ is denoted by $\mu\colon A\to A_{mon}.$

\begin{prop}
	Let $A$ be a partial magma and $f\colon A\to B$ be a homomorphism to an abelian monoid $B.$
	Then there exists a unique homomorphism $f_{mon}\colon A_{mon} \to B$
	such that $f_{mon}\circ \mu = f.$	
\end{prop}

\begin{proof}
	We put $f_{mon}([a_1\dotplus \dots \dotplus a_r]) = f(a_1) + \dots + f(a_r),$ then we have a homomorphism
	$f_{mon} \colon A_{mon} \to B$ such that $f_{mon}\circ \mu = f,$ which is unique.
\end{proof}

Remark that $\mu$ is not necessarily a monomorphism. Indeed, $\mu(A)$ has a natural structure of 
partial submagma of $A_{mon}$ which is weakly associative and is denoted by $A_{wass}.$

\subsection{Associative closure}

In this section, we show that for any partial magma $A,$ there exists 
a partial monoid $A_{ass}$ and a homomorphism $\alpha\colon A\to A_{ass}$
which is universal among the homomorphisms from $A$ to partial monoids.

Let $A$ be a partial magma and $\Set{B^{(\lambda)}}$ be a family of partial submagmas of $A.$
If we put 
\[ B = \cap_\lambda B^{(\lambda)} \mbox{~and~} B_2 = \cap_\lambda B^{(\lambda)}_2 \]
then $B$ is the largest partial submagma of $A$ which is contained in every $B^{(\lambda)}$'s.
If all the $B^{(\lambda)}$'s are partial submonoids of $A,$ then $B$ is also a partial submonoid of $A.$

On the other hand, let $A$ be a partial magma and $\Set{B^{(\lambda)}}$ be a family of 
partial submagmas of $A$ which is totally ordered by containment.
If we put 
\[ B = \cup_\lambda B^{(\lambda)} \mbox{~and~} B_2 = \cup_\lambda B^{(\lambda)}_2 \]
then $B$ is the smallest partial submagma of $A$ which contains every $B^{(\lambda)}$'s.
If all the $B^{(\lambda)}$'s are partial submonoids of $A,$ then $B$ is also a partial submonoid of $A.$

Let $A$ be a partial monoid and $B$ be a partial submagma of $A.$ We define $B_{ass, A}$ to be
the smallest partial submonoid of $A$ which contains $B.$ We call $B_{ass,A}$ the associative
closure of $B$ in $A.$ $B_{ass, A}$ can be constructed inductively as follows:

Put $B^{(0)} = B$ and $B^{(0)}_2 = B_2.$
Suppose we have constructed a partial submagma $B^{(n-1)}\subseteq A.$
Consider a condition
\[
	(*)~~~(a+b) + c \mbox{~can be calculated in~} B^{(n-1)}
\]
for a triple $(a,b,c) \in B^{(n-1)}\times  B^{(n-1)}\times  B^{(n-1)}.$
If we put
\begin{align*}
	B^{(n)} &= B^{(n-1)} \cup \Set{ b+c | (a,b,c) \mbox{~satisfies~}(*) }\mbox{~and~}\\
	B^{(n)}_2 &=  B^{(n-1)}_2\cup (\Set{0}\times B^{(n-1)} )\cup  (B^{(n-1)}\times \Set{0} )\\
	&\phantom{=} \cup \Set{ (b,c), (c,b), (a,b+c), (b+c, a) | (a,b,c)\mbox{~satisfies~}(*)},
\end{align*}
then $B^{(n)}$ has a natural structure of partial submagma of $A.$
Constructing $B^{(n)}, n=0,1,2,\dots $ inductively, we put
\[
	B' = \cup_{n\geq 0} B^{(n)} \mbox{~and~} B'_2 = \cup_{n\geq 0} B^{(n)}_2.
\]
Now, $B'$ is a partial submonoid of $A$ which contains $B$ and which is the smallest. Thus $B' = B_{ass, A}.$

For any partial magma $A,$ we define $A_{ass}$ to be the smallest partial submonoid of $A_{mon}$ which
contains $A_{wass}.$ The composite $A\to A_{wass} \to A_{ass}$ is denoted by $\alpha.$

\begin{prop}
	Let $A$ be a partial magma and $f\colon A\to B$ be a homomorphism to a partial monoid $B.$
	Then there exists a unique homomorphism $f_{ass}\colon A_{ass} \to B$
	such that $f_{ass}\circ \alpha = f.$	
\end{prop}


\subsection{Limits and colimits in  $\pmag$ and $\pmon$} \label{sec:pmag_pmon_complete_cocomplete}
It is easily checked that $\pmag$ has all small limits and all small colimits.
It is also easily checked that in $\pmon,$ all limits in $\pmag$ are also limits in $\pmon,$
so $\pmon$ has all small limits.
If we construct a colimit in $\pmag$ from a given diagram in $\pmon,$ then it is not 
a colimit in $\pmon,$ but composing with the functor which takes associative closure
makes it a colimit in $\pmon.$ This shows that $\pmon$ has all small colimits.

\subsection{Tensor product and $\Hom$}

In this section, we define tensor product and $\Hom.$
Propositions are given without proof since each of them 
can be proved by a formal argument.

Let $A, B$ be partial monoids and $\NN[A\times B]$ be the free abelian monoid generated by
the set $A\times B.$
Let $\sim$ be the equivalence relation on $\NN[A\times B]$ generated by 
\begin{enumerate}
	\item $(0, b) \dotplus x \sim x \sim (a,0)\dotplus x$ for all $a\in A, b\in B$ and $x\in \NN[A\times B],$
	\item $(a_1, b) \dotplus (a_2, b) \dotplus x \sim (a_1+a_2, b) \dotplus x$ for all $x\in \NN[A\times B], b\in B$ and $(a_1, a_2) \in A_2,$
	\item $(a, b_1) \dotplus (a, b_2) \dotplus x \sim (a, b_1+b_2) \dotplus x$ for all $x\in \NN[A\times B], a\in A$ and $(b_1, b_2) \in A_2.$
\end{enumerate}
We put $T(A, B) = \NN[A\times B]/ \sim.$ Then $T(A,B)$ has an abelian monoid structure such that
$\pi\colon \NN[A\times B] \to T(A,B)$ is a homomorphism of monoids. 
An element of $T(A,B)$ represented by $(a,b) \in A\times B$ is denoted by $a\otimes b.$
We give $\pi(A\times B) \subseteq T(A,B)$ the maximal partial magma structure.

\begin{defn}[tensor product]
	Let $A$ and $B$ be partial monoids. 
	The associative closure of $\pi(A\times B)$ is
	denoted by $A\otimes B$ and is called the {\bf tensor product} of $A$ and $B.$
\end{defn}
\begin{defn}[bilinear map]
We say that a map $f\colon A\times B \to C$ is bilinear if for each $a\in A$ the map $B\to C,$ given by 
$b\mapsto f(a,b)$ is a partial monoid homomorphism and for each $b\in B$ the map $A\to C$ given by
$a\mapsto f(a,b)$ is a partial monoid homomorphism.
\end{defn}

\begin{prop}
	Let $A, B, C$ be partial monoids and $f\colon A\times B \to C$ a biliear map.
	Then there exists a unique partial monoid homomorphism 
	$\tilde{f}\colon A\otimes B \to C$ which makes the following diagram commute:
	\begin{center}
	\begin{tikzcd}
		A\times B \ar{r}{f}\ar{d} & C\\ 
		A\otimes B\ar{ru}[below]{\tilde{f}},
	\end{tikzcd}
	\end{center}
	where the vertical map is the canonical map.\qed
\end{prop}

For partial magmas $A,B$ we would like to define
\begin{align*}
	\Hom_\pmag (A, B) &= \Set{ f\colon A\to B | f : \mbox{homomorphism} }\\
	\Hom_\pmag (A, B)_2 &= \Set{ (f,g) | (f(a), g(a)) \in B_2 \mbox{~for all~} a\in A }.
\end{align*}
But a formula $(f+g)( a ) = f(a) + g(a)$ does not define a homomorphism $f+g \colon A\to B$
unless $B$ is associative. If we assume that $B$ is a partial monoid, then $\Hom_\pmag (A,B)$
given above is a partial monoid. Thus we have a functor
$\Hom_\pmag( -, - ) \colon \pmag^{op}\times \pmon \to \pmon.$
We also have a functor 
$\Hom_\pmon( -, - ) \colon \pmon^{op}\times \pmon \to \pmon.$

\begin{prop}\label{prop:bilinear_pmag_pmon}
	Let $A$ and $B$ be partial magmas and $C$ be a partial monoid.
	If $f \colon A\times B\to C$ is a bilinear map, then there exists a unique bilinear map
	$\tilde{f}\colon A_{ass}\times B_{ass} \to C$ which makes the following diagram commute:
	\begin{center}
	\begin{tikzcd}
		A\times B \ar{r}{f}\ar{d} & C\\ 
		A_{ass}\times B_{ass}\ar{ru}[below]{\tilde{f}}.
	\end{tikzcd}
	\end{center}\qed
\end{prop}

\begin{prop}
	Let $A$ be a partial monoid, then
	we have an adjoint pair of functors
	\[
		\otimes A \colon \pmon \leftrightarrows \pmon : \Hom(A, -) \qedhere
	\]
	\qed
\end{prop}
\begin{prop}
	$(\pmon, \otimes , \fun)$ constitute a symmetric monoidal category. \qed
\end{prop}

\subsection{Equivalence relations}

For a detailed discussion on equivalence relations, see \S 2.5 of \cite{borceux-handbook-2} or \cite{nlab-congruences}. 
As opposed to the terminology in \cite{nlab-congruences}, we reserved the word ``congruences'' for effective equivalence relations on partial rings,
that apear in a later section of this paper.
Let $A$ be a partial magma and $R$ be an equivalence relation on $A$ in $\pmag.$
Thus $R$ is a partial submagma of $A\times A$ such that for any partial magma $X,$
\[
	\Hom_{\pmag} (X, R) \subseteq \Hom_{\pmag} (X, A\times A) 
	\cong \Hom_{\pmag}(X, A)\times \Hom_{\pmag}(X, A)
\]
is an equivalence relation on the set $\Hom_{\pmag} (X,A).$
Recall that $R$ is said to be effective if 
$R\rightrightarrows A$ is a kernel pair of some $f\colon A\to B,$ {\it i.e.}
the following diagram is a pullback diagram:
\[
	\begin{tikzcd}
		R\ar{r}\ar{d} & A\ar{d}{f}\\
		A\ar{r}[below]{f} & B.
	\end{tikzcd}
\]

\begin{defn}
	Let $A$ be a partial magma.
	We say that an equivalence relation $R$ on $A$ is 
	{\bf additive} if $R_2 = (R\times R) \cap A_2.$
\end{defn}

If $R$ is an additive equivalence relation on a partial magma $A,$ then we can define a
partial magma by
\begin{align*}
	&A\dslash R = ( \mbox{~the quotient set of~} A\mbox{~by~} R\\
	&(A\dslash R)_2 = \Set{ ([a_1], [a_2] ) | (a_1, a_2) \in A_2 }\\
	&[a_1]+[a_2] = [a_1+a_2].
\end{align*}

It is easily checked that the following diagram is a coequalizer diagram in $\pmag:$
\begin{center}
	\begin{tikzcd}
		R \ar[shift left = 0.3em]{r}\ar[shift right = 0.3em]{r}& A \ar{r} & A\dslash R.
	\end{tikzcd}
\end{center}

\begin{prop}\label{prop:effective-additive-pmag}
	An equivalence relation $R$ on a partial magma $A$ is effective if and only if 
	it is additive.
\end{prop}
\begin{proof}
Suppose that $R$ is effective. Then $R\rightrightarrows A$ is a kernel pair of 
some $f\colon A\to B.$ 
Let $C=\Set{ 0, a, b, a+b}$ be a partial monoid in which $a+b$ is the only non-trivial sum.
If $a_1Rb_1, a_2Rb_2$ and $(a_1, a_2), (b_1, b_2) \in A_2$, we define $h_i\colon C\to A$
by $h_i(a) = a_i, h_i(b) = b_i\,(i=1,2).$ Since $f\circ h_1 = f\circ h_2$ we have
a morphism $h\colon C\to R$ such that $h(a) = (a_1, a_2)$ and $h(b) = (b_1, b_2)$ by
universality of the pullback. Then $((a_1, a_2) , (b_1, b_2))\in R_2,$ which means that
$R$ is additive.

Conversely, suppose that $R$ is additive. Let $\pi \colon A\to A\dslash R$ be 
the projection morphism. To show that the diagram
\[
	\begin{tikzcd}
		R\ar{r}\ar{d} & A\ar{d}{f}\\
		A\ar{r}[below]{f} & A\dslash R
	\end{tikzcd}
\]
is a pullback diagram, let $g_1, g_2\colon B\to A$ be two morphisms such that
$f\circ g_1 = f\circ g_2.$ It is easily checked that $(g_1,g_2) \colon B\to A\times A$
factors through a morphism $B\to R$ since $R$ is additive. This proves that
$R$ is effective.
\end{proof}

Let $A$ be a partial monoid and $R$ be an additive equivalence relation on $A.$
Let $A/R = (A\dslash R)_{ass}$ be the associative closure of $A\dslash R.$
Then it is easily checked that the following diagram is a coequalizer diagram in $\pmon:$
\begin{center}
	\begin{tikzcd}
		R \ar[shift left = 0.3em]{r}\ar[shift right = 0.3em]{r}& A \ar{r} & A/R.
	\end{tikzcd}
\end{center}

\begin{prop}
	An equivalence relation $R$ on a partial monoid $A$ is effective if and only if 
	it is additive and the canonical map $\alpha \colon A\dslash R \to A/R$ is injective.
\end{prop}
\begin{proof}
Suppose that $R$ is effective. Then we can show that $R$ is additive by the same
argument as in the proof of Proposition \ref{prop:effective-additive-pmag}.
It follows that $R\rightrightarrows A$ is the kernel pair of the canonical morphism 
$\pi \colon A\to A/R.$
We need to show that $\alpha$ is injective. Assume, on the contrary, that
$\alpha$ is not injective and there exist $a_1, a_2 \in A$ such that $[a_1] \neq [a_2]$
in $A\dslash R$ but $\alpha([a_1]) = \alpha([a_2]).$ Let $f_i \colon \fun \to A$ 
be a morphism given by $f_i (1) = a_i\, (i=1,2).$ Since $R\rightrightarrows A$ is the kernel pair
of $\pi \colon A\to A/R,$ there exists a morphism $f\colon \fun \to R$ such that 
$f(1) = (f_1(1), f_2(1)) = (a_1, a_2),$ a contradiction. This proves that $\alpha$ is injective.

Conversely, suppose that $R$ is additive and the canonical map 
$\alpha \colon A\dslash R \to A/R$ is injective.
To show that $R\rightrightarrows A$ is the kernel pair of $\pi\colon A\to A/R,$
let $f_1, f_2\colon B\to A$ be morphisms such that $\pi \circ f_1 = \pi \circ f_2.$
If $\pi'\colon A\to A\dslash R$ denotes the canonical morphism, 
$\pi'\circ f_1 = \pi'\circ f_2,$ since $\alpha$ is injective.
Since $R\rightrightarrows A$ is the kernel pair of $\pi'$ in $\pmag,$
we have a unique morphism $f\colon B\to R$ such that $f = (f_1, f_2)$ in $\pmag.$
Since $B$ and $R$ are partial monoids, $f$ is a morphism in $\pmon,$
which is uniquely determined.
\end{proof}
\begin{cor}
	Let $R$ be an additive equivalence relation on a partial monoid $A.$
	The $R$ is effective if and only if the following condition holds:
	\[
		\mbox{(Condition)} \left\{
		\begin{array}{l}
		\mbox{~if~} aRa', bRb', cRc', (a,b) \in A_2, (b',c')\in A_2,\\
		(a+b)Rx, (b'+c')Rx', (x,c) \in A_2 \mbox{~and~} (a',x') \in A_2\\
		\mbox{~then~}(x+c) R (a'+x').
		\end{array}\right.
	\]
\end{cor}
\begin{cor}
	If $\Set{R_\lambda}$ is a family of effective equivalence relations on a partial monoid $A.$
	Then $\cap_\lambda R_{\lambda}$ is an effective equivalence relation on $A.$
\end{cor}

The next theorem is not used in the following part of this paper.

\begin{thm}
	The category of partial magmas is regular.
\end{thm}

\begin{proof}
	As noted in \S\ref{sec:pmag_pmon_complete_cocomplete}, $\pmag$ is 
	complete and cocomplete. To show that a pullback of a regular epimorphism 
	is a regular epimorphism, let $f\colon A\to B$ be a regular epimorphism.
	If $K$ is the kernel pair of $f,$ then it is readily checked that $B \cong A\dslash K.$
	So we may assume that $B = A\dslash K$ and $f=\pi \colon A\to A\dslash K.$
	Let 
	\[
		\begin{tikzcd}
			P\ar{r}\ar{d}{h} & A \ar{d}{\pi}\\
			C \ar{r}{g} & A\dslash K
		\end{tikzcd}
	\]
	be the pullback of $\pi$ along a morphism $g\colon C\to A\dslash K.$
	Since $\pi$ is a surjective map, so is $h.$ 
	Let $L$ be the kernel pair of $h.$ 
	We show that the diagram
	\[
		\begin{tikzcd}
			L \ar[shift left = 0.3em]{r}{l_1}\ar[shift right = 0.3em]{r}[below]{l_2}& P \ar{r} & C.
		\end{tikzcd}
	\]
	is a coequalizer diagram. For that purpose, suppose that there exists a morphism 
	$m\colon P\to D$ such that $m\circ l_1 = m\circ l_2.$ Since $h$ is surjective,
	we have a map $\tilde{m}\colon C\to D$ of sets which satisfies $\tilde{m} \circ h= m.$
	To show that $\tilde{m}$ is a homomorphism of partial magmas, suppose
	$(c_1, c_2) \in C_2.$ Then $(g(c_1), g(c_2)) \in (A\dslash K)_2.$ So we can take a summable
	pair $(a_1, a_2) \in A_2$ such that $(\pi(a_1), \pi(a_2)) =(g(c_1), g(c_2)).$
	Thus, $((a_1, c_1) , (a_2, c_2)) \in P_2,$ which implies that 
	$(\tilde{m}(c_1), \tilde{m}(c_2)) = (m(a_1, c_1) , m(a_2, c_2)) \in D_2.$
	Therefore $\tilde{m}$ is a unique homomorphism which satisfies $\tilde{m} \circ h= m.$
\end{proof}
\section{Partial Rings}

\subsection{Partial Rings}
\begin{defn}[partial ring]
	A {\bf partial ring} is a partial monoid $A$ with a bilinear operation 
	$\cdot,$ called a multiplication which is associative, commutative,
	and have a unit $1.$ Here, bilinearity of $\cdot$ is, by definition, equivalent to the following 
	condition:
	\begin{enumerate}
		\item $0\cdot a = 0$ for all $a\in A$ and
		\item $(a_1, a_2) \in A_2$ implies $(a_1\cdot x, a_2\cdot x) \in A_2$
		and $a_1\cdot  x + a_2\cdot  x = (a_1+a_2)\cdot x$ for all $x\in A.$
	\end{enumerate}
\end{defn}

\begin{defn}
	Let $A$ and $B$ be partial rings. A map $f\colon A\to B$ is a 
	homomorphism of partial rings if it is a homomorphism of underlying partial monoids and
	satisfies	$f(1) = 1$ and $f(ab) = f(a)f(b)$ for all $a, b\in A.$
	The category of partial rings are denoted by $\pring.$
\end{defn}

\begin{rem}
	A partial ring is nothing but a commutative monoid object in the symmetric monoidal category
	$(\pmon, \otimes , \fun).$
\end{rem}

\begin{example}
	A partial ring of order 2 is isomorphic to one of 
	$\fun, \FF_2$ and $\BB.$
\end{example}

\subsection{A partial ring given by generators and a summability list}
	In this section let $n$ be a fixed positive integer and $N=\NN[x_1,\dots, x_n]$ be the set of the polynomials of 
	indeterminates $x_1, \dots, x_n$ with coefficients in $\NN.$
	Let $S = \Set{s_1, \dots, s_r }$ be a subset of $N.$
	Then let
	\[
		\fun \angles{x_1,\dots, x_n | \exists s_1, \dots, s_r }
	\]
	denote the smallest partial subring of $N$ which contains every subsum of
	one of $s_1, \dots , s_r,$ and in which every pair $(a,b)$
	such that $a+b$ is a subsum of one of $s_1, \dots , s_r$ is summable.
	
	If $c_1, \dots, c_n$ are elements of a partial ring $A,$
	and $s \in N,$ then $s(c_1, \dots, c_n) \in A$ denotes the value of the polynomial $s$ as usual.
	Let $s = m_1 + \dots + m_r$ be the unique factorization of $s$ into a sum of monomials with coefficients 1.
	We say that $s(c_1, \dots, c_n)$ {\bf can be calculated in $A$} if $(m_{i_1}, \dots, m_{i_s}) \in A_s$ for
	all subsets $\set{i_1,\dots , i_s} \subseteq \set{1, \dots, r}.$
	
	\begin{prop}\label{prop:pring_hom_by_generators}
		Let $c_1, \dots, c_n$ be elements of a partial ring $A$ and $S = \Set{s_1, \dots, s_r }$ be a subset of $N.$
		If $s_i (c_1, \dots, c_n)$ can be calculated in $A$ for all $i=1,\dots , r, $ then there exists a unique
		partial ring homomorphism
		\[
			\varphi \colon \fun \angles{x_1,\dots, x_n | \exists s_1, \dots, s_r }\to A
		\]
		such that $\varphi(x_i) = c_i\,(i=1,\dots, n).$
	\end{prop}

	The remainder of this section is devoted to a proof of the above proposition.
	Let $W$ be a partial submagma  of the underlying partial monoid of $N.$ 
	We say that an element $w\in W$
	{\bf is factorized in $W$} if the unique factorization $w = m_1 + \dots + m_r$
	in $N$ can be calculated in	$W$ under some appropriate reordering and 
	supplement of parentheses.
	We say that $W$ {\bf has the factorization property} if every element of $W$ is
	factorized in $W.$
	
	\begin{lem}
		Let $B \subseteq N$ be a partial submagma of the underlying partial monoid of $N.$ 
		If $B$ has the factorization property,
		then so is its associative closure $B_{ass, N}$ in $N.$
	\end{lem}
	
	\begin{proof}
		Recalling the inductive construction of the associative closure,
		it is sufficient to show that if $B^{(n-1)}$ has the factorization property then so is $B^{(n)}.$
		So assume that $B^{(n-1)}$ has the factorization property.
		New elements in $B^{(n)}$ are of the form $b+c$ where there exists $a\in B^{(n-1)}$ such that
		$(a+b)+c$ can be calculated in $B^{(n-1)}.$
		Let $b= m_1 + \dots + m_r$ and $c = m'_1 + \dots + m'_s$ be the unique factorization of $b$ and $c$
		in $N.$ By assumption, there exists a way to supply parentheses to these formula so that they can be
		calculated in $B^{(n-1)}.$ 
		Using these supplement of parentheses, $b+c$ can be calculated in $B^{(n)}.$
		Thus $B^{(n)}$ has the factorization property.
	\end{proof}

	\begin{lem}\label{lem:inc_seq_of_submonoids_unique_factorization}
		There exists an increasing sequence 
		$X^{(0)}\subseteq X^{(1)}\subseteq  \dots$ of partial submonoids of $N$
		such that
		\begin{enumerate}
			\item $X^{(i)}$ has the factorization property for all $i$,
			\item if $a,b \in X^{(i)}$ then $ab \in X^{(i+1)}$ for all $i$ and
			\item $\fun \angles{x_1,\dots, x_n | \exists s_1, \dots, s_r } = \cup_{i\geq 0} X^{(i)}$ as a partial monoid.
		\end{enumerate}
	\end{lem}
	\begin{proof}
		If we put
		\begin{align*}
			Y^{(0)} &= \Set{ 0, 1 } \cup \Set{ s | s \mbox{~is a subsum of some~}s_i \in S},\\ 
			(Y^{(0)})_2 &= \Set{ (a,b) | a+b \mbox{~is a subsum of some~}s_i \in S} 
		\end{align*}
		then this is a partial submagma of $N,$ which has the factorization property.
		Let $X^{(0)} = (Y^{(0)})_{ass, N}$ be its associative closure in $N.$ By the previous lemma,
		$X^{(0)}$ has the factorization property.
		Suppose that we have constructed $X^{(0)}, \dots, X^{(k-1)}$ such that $X^{(i)}$ has the factorization
		property for $i=0,\dots, k-1$ and  if $a,b \in X^{(i)}$ then $ab \in X^{(i+1)}$ for all $i = 0, \dots, k-2.$
		The we put
		\begin{align*}
			Y^{(k)} &= X^{(k-1)}\cup \Set{ ab | a, b \in X^{(k-1)} },\\ 
			(Y^{(k)})_2 &= (X^{(k-1)})_2 \cup \Set{ (a_1 b, a_2 b) | (a_1, a_2) \in (X^{(k-1)})_2 , b\in  X^{(k-1)} }.
		\end{align*}
		$Y^{(k)}$ is a partial submagma of $N.$ To show that $Y^{(k)}$ has the factorization property,
		let $a= m_1 + \dots + m_r$ and $b = m'_1 + \dots + m'_s$ be the unique factorization of $a$ and $b$
		in $N.$ By assumption, there exists a way to supply parentheses to these formula so that they can be
		calculated in $X^{(k-1)}.$ Then the unique factorization of $ab$ is given by
		$ab = \sum m_i m'_j.$ Now, 
		supply parentheses to $ab =  m_1 b + \dots + m_r b$ in the way that we supplied parentheses
		to the formula $a = m_1 + \dots + m_r.$ This formula can be calculated in 
		$Y^{(k)},$ once $b$ is calculated. So we supply parentheses to each 
		$b = m'_1 + \dots + m'_s$ so that it is calculated in $Y^{(k)}.$ Thus $Y^{(k)}$ has the factorization property.
		Let $X^{(k)} = (Y^{(k)})_{ass, N}$ be its associative closure in $N.$ By the previous lemma,
		$X^{(k)}$ has the factorization property.
		
		We have constructed an increasing sequence $X^{(0)}\subseteq X^{(1)}\subseteq  \dots$ 
		of partial submonoids of $N$ which satisfies (1) and (2). Now it is clear that it satisfies (3).	
	\end{proof}

	\begin{proof}[Proof of Proposition \ref{prop:pring_hom_by_generators}]
		Let $Y^{(k)}$ and $X^{(k)}$ be as in (the proof of) 
		Lemma \ref{lem:inc_seq_of_submonoids_unique_factorization}.
		The assumption of the proposition that $s_i(c_1, \dots, c_n)$ can be calculated in $A$ for all $i$
		is equivalent to the condition that a partial magma homomorphism
		$\varphi\colon Y^{(0)} \to A$ can be defined by $\varphi(x_i) = c_i.$ Suppose we have shown that
		$\varphi$ extends to a partial magma homomorphism $\varphi \colon Y^{(k-1)} \to A.$
		If $(a + b) + c$ can be calculated in $Y^{(k-1)},$ then $(\varphi(b) , \varphi(c)) \in A_2,$
		 since $A$ is a
		partial monoid. If we put $\varphi(b+c) = \varphi(b)+\varphi(c),$
		then this is well-defined since $Y^{(k-1)}$ has the factorization property.
		In this way, we can extend $\varphi$ uniquely to 
		a partial monoid homomorphism $\varphi \colon X^{(k-1)}\to A.$ 
		Let $a, b \in X^{(k-1)}.$ We put $\varphi(ab) = \varphi(a)\varphi(b).$ This is well-defined by
		the factorization property of $Y^{(k)}.$ So $\varphi$ uniquely extends to
		a partial magma homomorphism $\varphi\colon Y^{(k)} \to A.$
		Continuing this process inductively, $\varphi$ extends uniquely to a partial monoid
		$\varphi \colon \cup_{k} X^{(k)} \to A,$ which is clearly a partial ring homomorphism
		$\varphi \colon \fun \angles{x_1,\dots, x_n | \exists s_1, \dots, s_r }\to A.$
	\end{proof}
	\subsection{Congruences}

	\begin{defn}[Congruence]
		Let $A$ be a partial ring. By a {\bf congruence} on $A,$ we mean an effective equivalence
		relation on $A.$ 
	\end{defn}

	\begin{prop}
		Let $A$ be a partial ring.
		A partial subring $R \subseteq A\times A$ is a congruence on $A$
		if and only if its underlying partial monoid is an effective equivalence relation
		on the underlying partial monoid of $A.$
	\end{prop}
	\begin{proof}
		Assume that $R$ is a congruence on $A.$
		Then $R\rightrightarrows A$ is a kernel pair of some morphism $f\colon A\to B$ in
		$\pring.$ Since a limit of any diagram 
		in $\pring$ is constructed by taking a limit of the corresponding
		diagram in $\pmon,$ $R\rightrightarrows A$ is the kernel pair of the 
		underlying homomorphism $f\colon A\to B$ in $\pmon.$ So the underlying
		partial monoid of $R$ is an effective relation on the underlying partial monoid of $A.$
		
		Conversely, assume that the underlying
		partial monoid of $R$ is an effective relation on the underlying partial monoid of $A.$
		We can define a map $A\dslash R \times A\dslash R \to A\dslash R$ by
		$([a_1], [a_2]) \mapsto [a_1 a_2],$ since $R$ is a partial subring of $A\times A.$
		Then $m'$ is bilinear since the multiplication map of $A$ is bilinear.
		Then by Proposition \ref{prop:bilinear_pmag_pmon}, $m'$ induces 
		a bilinear map $\bar{m}\colon A/R\times A/R \to A/R.$
		This makes $A/R$ into a partial ring.
		It is clear that the canonical morphism $\pi\colon A\to A/R$ is 
		a homomorphism of partial rings.
		To show that $R$ is the kernel pair of $\pi\colon A\to A/R,$
		let $f_1, f_2 \colon B\to A$ be two morphisms such that
		$\pi\circ f_1 = \pi\circ f_2.$ 
		The underlying partial monoid of $R$ is the kernel pair of $\pi$ in $\pmon,$
		there exists a unique morphism $f\colon B\to A$ of underlying partial monoids
		such that $f = (f_1, f_2).$ 
		Since $R$ is a partial subring of $A\times A,$ $f$ is a homomorphism of partial rings.
	\end{proof}
	\begin{cor}
		If $\Set{R_\lambda}$ is a family of congruences on a partial ring $A,$
		then $\cap_{\lambda} R_{\lambda}$ is a congruence on $A.$
	\end{cor}
	\begin{cor}
		Let $A$ be a partial ring.
		For any subset $S$ of $A\times A,$ there exists the smallest congruence on $A$
		which contains $S.$ (It is denoted by $\angles{S}.$)
	\end{cor}

	\begin{prop}
		Let $R$ be a congruence on a partial ring $A.$ Then 
		for any homomorphism $f\colon A\to B$
		for which $f(a_1) = f(a_2)$ for all $(a_1, a_2) \in R,$ there exists
		a unique homomorphism $\tilde{f} \colon A/R \to B$ which makes the following diagram commute:
		\begin{center}
			\begin{tikzcd}
				A \ar{r}{f}\ar{d}[left]{\pi} & B\\
				A/R \ar{ru}[below]{\tilde{f}}.
			\end{tikzcd}
		\end{center}
	\end{prop}

	\begin{proof}
		Let $f\colon A\to B$ be a morphism of partial rings such that
		$f(a_1) = f(a_2)$ for all $(a_1, a_2) \in R.$ Since $A\to A\dslash R$ is a coequalizer of
		$R\rightrightarrows A$ in $\pmag$ we have a homomorphism 
		$A\dslash R \to B$ of partial magmas. Similarly, we have a homomorphism $A/R \to B$
		of partial monoids, and we have a commutative diagram
		\begin{center}
			\begin{tikzcd}
				A\ar{r}\ar{rd}[below]{f} & A\dslash R \ar{r}\ar{d} & A/R\ar{ld}{\tilde{f}}\\
				& B.
			\end{tikzcd}
		\end{center}
		In the following diagram, the largest rectangle is commutative, since
		$f$ is a homomorphism of partial rings and by the definition of $m' :$
		\begin{center}
			\begin{tikzcd}
				 A\dslash R \times A\dslash R\ar{r}\ar{d}{m'} 
				& A/R \times A/R \ar{d}{\bar{m}} \ar{r} & B\times B \ar{d}{m_B}\\		
				A\dslash R\ar{r}  &  A/R \ar{r}& B.
			\end{tikzcd}
		\end{center}
		Also, the left small rectangle is commutative by the definition of $\bar{m}.$
		Then so is the right small rectangle, by Proposition \ref{prop:bilinear_pmag_pmon}.
	\end{proof}

\subsection{Ideals}

\begin{defn}[ideal]
	An {\bf ideal} of a partial ring $A$ is a partial submonoid $I$ of $A$
	 such that $I_2 = A_2 \cap (I\times I)$
	and $ax \in I $ for any $a\in A$ and $x\in I.$
\end{defn}

\begin{example}
	Let $T$ be a subset of $A.$ If we put
	\begin{align*}
		I &= \Set{ a_1 t_1 + \dots + a_r t_r | \begin{array}{l}r\in \NN, a_i \in A, t_i \in T  \\
		(a_1 t_1 , \dots , a_r t_r)\in A_r
		\end{array}
		},
	\end{align*}
	then $I$ is the smallest ideal which contains $T.$ This ideal $I$ is denoted by $(S).$
	If $ S = \Set{a}$ is a singleton, $(S)$ is also denoted by $(a).$
\end{example}

Let $\ia$ and $\ib$ be two ideals of $A.$ The smallest ideal which contains $\ia$ and $\ib$ is denoted by
$\ia + \ib.$ On the other hand, if we put 
	\begin{align*}
		I &= \Set{ a_1 b_1 + \dots + a_r b_r | \begin{array}{l}r\in \NN, a_i \in A, b_i \in B  \\
		(a_1 b_1 , \dots , a_r b_r)\in A_r
		\end{array}
		},
	\end{align*}
then $I$ is an ideal which is contained in both $\ia$ and $\ib.$ This ideal $I$ is denoted by $\ia \ib.$

Let $\varphi\colon A\to B$ be a homomorphism of partial rings and $J\subseteq B$ be an ideal. 
If we put
\[ I = \varphi^{-1}(J) \]
then $I$ is an ideal of $A.$ This ideal $I$ is denoted by $\varphi^*(J).$ On the other hand, let
$I$ be an ideal of $A.$ The smallest ideal of $B$ which contains $\varphi(I)$ is denoted by $\varphi_*(I).$

\begin{defn}[prime ideal]
	An ideal $I$ of a partial ring $A$ is called a {\bf prime ideal} if $I\neq A$ and
	$ab \in I$ implies $a\in I$ or $b\in I$ for any $a,b \in A.$
\end{defn}

\subsection{Localization}

Let $S$ be a multiplicative subset of $A.$ We put
\begin{align*}
	S^{-1} A &= \Set{ a/s | a\in A, s\in S} \\
	(S^{-1} A)_2 &= \Set{ (a/s, b/t) | \mbox{there exists~} u \in S \mbox{~s.t.~}(uta, usb) \in A_2},
\end{align*}
where $a/s$ denotes the usual equivalence class such that $a/s = b/t$ if and only if
there exists $u\in S$ such that $uta = usb.$ 
It is readily checked that $S^{-1}A$ is a partial ring by putting $\frac{a}{s}\frac{b}{t} = \frac{ab}{st}.$ 
The homomorphism $\lambda \colon A\to S^{-1}A$ given by $\lambda(a) = a/1$ has the universal property
of localization.

\begin{defn}[local pring]
	A partial ring $A$ is called a {\bf local partial-ring} (local pring for short)
	 if it has a unique maximal ideal.
\end{defn}

If $\ip$ is a prime ideal of a partial ring $A,$ then the localization $A_\ip$ of $A$
by the multiplicative set $A\setminus \ip$ is a local pring.

If $A , B$ are local prings, a homomorphism $\varphi\colon A\to B$ is called a
{\bf homomorphism of local prings} if $\varphi^*(\im_B) = \im_A,$ where
$\im_A$ (resp. $\im_B$) is the maximal ideals of $A$ (resp. $B).$

\begin{prop}
Let $A$ be a partial ring and $\lambda\colon A\to S^{-1}A$ be the localization 
by a multiplicative subset $S$ of $A.$ 
For any ideal $I\subseteq A, \lambda^*\lambda_*(I) = I.$
\end{prop}

\begin{defn}[partial field]
	A {\bf partial field} is a partial ring in which every non-zero element is multiplicatively invertible.
\end{defn}

A partial ring is a partial field iff the ideal $(0)$ is a maximal ideal.

\subsection{Blueprints}

We can compare partial rings to Lorscheid's blueprints\cite{lorscheid-the-geometry-of-1}. 
As is explained in the above paper, Deitmar's sesquiads\cite{deitmar-congruence-schemes}
are special kinds of blueprints. 
In this section, we show that 
partial rings are also special kinds of blueprints as explained below.
 
Let $A$ be a partial ring and let $\NN[A]$ denote the monoid-semiring determined by
the underlying multiplicative monoid of $A.$
We put
\[
	R_0(A) = \Set{ (a_1\dotplus a_2, a ) | (a_1, a_2) \in A, a_1+a_2 = a}\cup \Set{(0,\emptyset)}.
\]
Let $R(A)$ be the smallest additive equivalence relation on $\NN[A].$ In other words,
$R(A)$ is the smallest equivalence relation on $\NN[A]$ which contains
\[
	R_1(A) = \Set{(0 \dotplus x, x) , (a_1 \dotplus a_2 \dotplus x, (a_1+a_2) \dotplus x ) | (a_1, a_2 ) \in A_2, x\in \NN[A]}.
\]

\begin{lem}\label{lem:pring_to_blueprint_to_pring}
	If $(a_1\dotplus a_2, a ) \in R(A),$ then $(a_1,a_2) \in A_2$ and $a_1+a_2 = a.$
\end{lem}
\begin{proof}
Consider the following property for $x = a_1\dotplus \dots \dotplus a_r\in \NN[A]$
\[
\mbox{(Property)}~~~\mbox{if~}\dotplus\mbox{~is replaced with~}+, x\mbox{~can be calculated in~}A.
\]
Since $A$ is a partial monoid,
for any $(x,y) \in R_1(A), x$ has this property if and only if $y$ has, and the sum for $x$
and that for $y$ are equal.
Then the same is true for any $(x,y) \in R(A).$ Suppose $(a_1 \dotplus a_2, a) \in R(A).$
Since $a$ can be calculated in $A,$ so is $a_1 \dotplus a_2$ and $a_1 + a_2 = a.$
\end{proof}

\begin{lem}
	$R(A)$ is multiplicative.
\end{lem}
\begin{proof}
	If $(\alpha, \beta) \in R_0(A)$ and $c\in A$ then $(\alpha c, \beta c) \in R_0(A).$
	Then it is readily checked that 
	\begin{enumerate}
		\item if $(\alpha, \beta) \in R(A)$ and $c\in A$ then $(\alpha c, \beta c) \in R(A),$
		\item if $(\alpha, \beta) \in R(A)$ and $\xi\in \NN[A]$ 
		then $(\alpha \xi, \beta \xi) \in R(A)$ and
		\item if $(\alpha, \beta) \in R(A)$ and $(\gamma, \delta)\in \NN[A]$ 
		then $(\alpha \gamma, \beta \delta) \in R(A),$
	\end{enumerate}
	which shows that $R(A)$ is multiplicative.
\end{proof}

It follows that $B(A) = (A, R(A))$ is a blueprint. 

Now, let $M$ be a commutative (multiplicative) monoid and $B = (M,R)$ be a blueprint.
Put 
\[
	R_0 = \Set{ (a_1\dotplus a_2, a ) \in R | a_1, a_2 , a \in M} \cup \Set{(0,\emptyset)}.
\]
\begin{defn}
A blueprint $(M, R)$ is {\bf generated by binary operations} if
$R$ is the smallest additive equivalence relation which contains $R_0.$
A blueprint $(M, R)$ is {\bf associative} if
$(a_1\dotplus a_2 \dotplus a_3, b)\in R, a_i\,(i = 1,2,3), b\in M$ 
then there exists $c\in M$ such that $(a_1\dotplus a_2, c) \in R.$
\end{defn}

\begin{lem}
Let $A$ be a partial ring.
The blueprint $B(A) = (A, R(A))$ is generated by binary operations and associative.
\end{lem}
\begin{proof}
	$B(A)$ is generated by binary operations by its construction.
	If $(a_1\dotplus a_2 \dotplus a_3, b)\in R, a_i\,(i = 1,2,3), b\in A,$
	then by Lemma \ref{lem:pring_to_blueprint_to_pring}, $a_1 + a_2 + a_3 = b$ in $A.$
	Then by the associativity of a partial monoid $A,$ there exists $c\in A$ 
	such that $a_1 + a_2 = c $ in $A.$ This means that $B(A)$ is associative.
\end{proof}

\begin{prop}
	The category of partial rings is isomorphic to the category of 
	proper cancellative and associative blueprints with zero which are 
	generated by binary operations.
\end{prop}

\begin{proof}
	Let $B = (A, R)$ be a proper cancellative and associative blueprint with zero which is 
	generated by binary operations.
	We put 
	\[
		A_2 = \set{(a_1, a_2) \in A^2 | \exists a\in A \mbox{~s.t.~}(a_1 \dotplus a_2 , a ) \in R}
	\]
	and define $a_1 + a_2 = a$ if $(a_1 \dotplus a_2 , a ) \in R.$ Since $B$ is a proper
	blueprint with a zero, this determines a partial magma structure on $A.$
	Since $B$ is associative, $A$ is a partial monoid. Finally, $A$ is a partial ring since
	$R$ is multiplicative.
	
	Conversely, if $A$ is a partial ring, let $\NN[A]$ denote the semiring determined by
	the underlying multiplicative monoid of $A.$ 
	
\end{proof}

\section{Partial Schemes}

\subsection{Locally PRinged Spaces}

\begin{defn}[locally pringed space]
	A {\bf locally partial-ringed space} (locally pringed space for short) is a pair $(X, \calO_X)$ where $X$ is a topological space and $\calO_X$ is a
	sheaf of partial rings on $X$ whose stalks are local prings.
	Let $(X,\calO_X)$ and $(Y,\calO_Y)$ be two locally pringed spaces.
	A morphism $(X, \calO_X) \to (Y,\calO_Y)$ is
	\begin{enumerate}
		\item a map $f\colon X\to Y$ of topological spaces and
		\item a homomorphism $f^\#\colon \calO_Y\to f_* \calO_X$ of sheaves of partial rings over $Y$
		which induces a homomorpshim of local prings $o_{Y,f(x)} \to o_{X,x}$ for all point $x\in X.$
	\end{enumerate}
\end{defn}

\subsection{Affine Partial Schemes}
	In this section, we will define affine partial schemes.
	Most statements and proofs in \S 3.1 and \S 3.2 of \cite{lorscheid-the-geometry-of-1} 
	about affine blue schemes can be read as statements and proofs about affine partial schemes.
	We will give proofs for some of them rather when it is simpler than those in 
	\cite{lorscheid-the-geometry-of-1} to show the simplicity of our theory.
	
	Let $A$ be a partial ring and
	$X_A$ be the set of prime ideals of $A.$
	For any $a\in A,$ let $D(a)$ be the set of prime ideals $\ip$ such that $a\notin \ip.$
	We give $X_A$ the topology generated by $D(a)$'s for all $a\in A.$ This topology is called the Zariski topology.
	For any ideal $\ia,$ let $V(\ia)$ denote the set of prime ideals which contain $\ia.$
	\begin{prop}
		If $\ia$ does not meet a multiplicative set $S,$ then there exists a prime ideal $\ip$
		which contains $\ia$ and does not meet $S.$
	\end{prop}
	\begin{proof}
		Let $\Sigma$ be a family of ideals which contain $\ia$ and do not meet $S.$
		Since $\ia$ is an element of it, $\Sigma$ is not empty. By the Zorn's lemma,
		there exists a maximal element $\ip$ of $\Sigma.$ Suppose that $ab \in \ip.$
		Then $(\ip + (a)) (\ip + (b)) \subseteq \ip.$
		Since $\ip$ does not meet $S,$ at least one of  $(\ip + (a))$ and $(\ip + (b))$
		does not meet $S.$ 
		Since $\ip$ is maximal with this property, one of $a$ and $b$ is an element of $\ip.$
		This proves that $\ip$ is a prime ideal.
	\end{proof}
	\begin{cor}
		If $D(b) \subseteq D(a),$ then there exist $k\in \NN$ such that $b^k \in (a).$
	\end{cor}
	\begin{proof}
		 Take $\ia = (a)$ and $S = \Set{b^k | k\in \NN}$ in the previous lemma.
	\end{proof}
	\begin{cor}
		If we put $\sqrt{\ia} = \Set{ a\in A | \exists r \in \NN\mbox{~s.t.~} a^r \in \ia},$ then
		\[
			\sqrt{\ia} = \bigcap_{\ia \subseteq \ip,~\ip\in X_A} \ip
		\]
	\end{cor}
	\begin{proof}
		It is clear that $\sqrt{\ia} \subseteq \bigcap \ip.$ For the converse, assume that
		$a \in A$ is not in $\sqrt{\ia},$ or equivalently, $\ia$ does not meet the multiplicative 
		subset $S = \Set{ a^r | r\in \NN}$ of $A.$ Then
		by proposition, there exists a prime ideal $\ip$ which contains $\ia$
		and does not meet $S.$ In particular, $a\notin \ip$ and this proves that
		$\sqrt{\ia} \supseteq \bigcap \ip.$
	\end{proof}
	\begin{prop}
		$D(a)$ is quasi-compact for all $a\in A.$ In particular $X_A$ is quasi-compact.
	\end{prop}
	
	\begin{proof}
		It is sufficient to show that 
		if $D(a) = \cup_{\lambda\in \Lambda} D(a_{\lambda})$ for some subset $\set{a_\lambda}\subseteq A,$ then 
		there exist finite elements $\lambda_1, \dots , \lambda_r \in \Lambda$ such that
		$D(a) = D(a_{\lambda_1})\cup \dots \cup D(a_{\lambda_r}).$
		
		So assume that 
		$D(a) = \cup_{\lambda\in \Lambda} D(a_{\lambda})$ for some subset $\set{a_\lambda}\subseteq A.$
		Then $a\in \ip \iff \Set{a_\lambda}\subseteq \ip$ for all $\ip \in X_A,$ which implies
		that $\sqrt{(a)} = \sqrt{(\Set{a_\lambda})}.$ So there exists $r\in \NN$ such that
		$a^r \in (\Set{a_\lambda}),$ which means that we can take some
		$a_{\lambda_1}, \dots , a_{\lambda_s}$ and $c_1, \dots, c_s \in A$ such that
		$(c_1a_{\lambda_1}, \dots , a_sa_{\lambda_s} )\in A_s$ and
		$a^r =c_1a_{\lambda_1}+ \dots + a_sa_{\lambda_s},$ so 
		$a \in \sqrt{(a_{\lambda_1}, \dots , a_{\lambda_s})}.$
		Then 
		$\Set{a_{\lambda_1}, \dots , a_{\lambda_s}} \subseteq \ip \implies a\in \ip$ for all 
		$\ip \in X_A$ and this means that 
		$D(a) \subseteq D(a_{\lambda_1}) \cup \dots \cup D(a_{\lambda_s}).$
	\end{proof}
	
	For any open set $U$ of $X= X_A,$ let $S_U$ be the subset of $A$ consisting of all $a\in A$ such that
	$a\notin \ip$ for all $\ip \in U.$ Then $S_U$ is a multiplicative subset of $A.$
	If we put $\calO'_X(U) = S_U^{-1} A,$ then $\calO'_X$ is a presheaf of partial rings on $X.$ Its sheafification is
	denoted by $\calO_X.$
	
	\begin{prop}
		 If $x = \ip\in X_A$ is a point, the stalk $o_{X,x}$ at $x$ coincides with the 
		local pring $A_\ip.$
	\end{prop}
	
	The locally pringed space $(X_A, \calO_X)$ constructed above is denoted by $\Spec (A).$
	
	\begin{defn}[affine pscheme]
		A locally pringed space is called an {\bf affine partial scheme} (affine pscheme for short)
		 if it is isomorphic
		as a locally pringed space to $\Spec A$ for a partial ring $A.$ 
	\end{defn}
	
	Let $A, B$ be partial rings and $\varphi\colon A\to B$ be a homomorphism.
	Suppose that $\Spec(A) = (X,\calO_X)$ and $\Spec(B) = (Y,\calO_Y).$

	\begin{prop}
		Let $A$ be a partial ring and $(X,\calO_X) = \Spec (A).$ 
		For any element $s\in A,$ there exists a monomorphism $A_s \to \calO_X(D(s)).$
	\end{prop}
	\begin{proof}
		We can define a homomorphism $\varphi\colon A_s \to \calO_X(D(s))$ by putting\\
		$\varphi(a/s^r) (\ip) = a/s^r \in A_{\ip}$ for all $\ip \in X.$
		Suppose $\varphi(a/s^r) = \varphi(b/s^q).$
		This means that $a/s^r = b/s^q$ in $A_\ip$ for all $\ip\in D(s),$ that is, we have $h\notin \ip$
		such that $hs^q a = hs^r b.$ 
		If we put
		\[
			\ia = \Set{ x\in A | xs^qa = xs^rb},
		\]
		then $\ia$ is an ideal of $A$ and $\ia \not\subseteq \ip$ for all $\ip \in D(s).$ Then we have that
		$\sqrt{(s)} \subseteq \sqrt{\ia}.$ If $s^k \in \ia,$ then $s^{k+q} a = s^{k+r} b, $ which means that
		$a/s^r = b/s^q$ in $A_s.$ Therefore $\varphi$ is injective.
	\end{proof}
	If we take $s=1$ in the above proposition, we obtain a monomorphism $\gamma \colon A\to \calO_X(X).$
	The following lemma, which is an extraction from the proof of Lemma 3.16 in \cite{lorscheid-the-geometry-of-1}
	is useful.
	\begin{lem}\label{lem:integralization}
		Let $A$ be a partial ring and $(X,\calO_X) = \Spec(A).$ We put $B = \calO_X(X).$ 
		For any element $\sigma \in B,$ there exists an element $s\in A$ such that $s\sigma \in \gamma(A).$		
	\end{lem}
	\begin{proof}
		We can prove this similarly as in the second paragraph of the proof of Lemma 3.16 of \cite{lorscheid-the-geometry-of-1}.
	\end{proof}
	\begin{thm}
		Let $A$ be a partial ring and $(X,\calO_X) = \Spec(A).$ We put $B = \calO_X(X)$ and let $(Y,\calO_Y) = \Spec (B).$
		Then $(X, \calO_X)$ and $(Y, \calO_Y)$ are isomorphic.
	\end{thm}
	\begin{proof}
		This theorem is only a mixture of Lemma 3.16 and Lemma 3.18 of \cite{lorscheid-the-geometry-of-1}.
		At the final step of the proof of
		Lemma 3.18 of \cite{lorscheid-the-geometry-of-1}, we can use Lemma \ref{lem:integralization}.
	\end{proof}
	
	The following corollary is immediate from the theorem.
	
	\begin{cor}\label{cor:one-one-correspondence}
		There exists a one-to-one correspondence between the morphisms $\Spec A\to \Spec B$ of 
		affine pschemes and the morphisms $B\to \Gamma ( \Spec A )$ of partial rings.
	\end{cor}
	
	The following definition is a translation into our case from \cite{lorscheid-the-geometry-of-1}.
	\begin{defn}
		A partial ring of the form $\Gamma(X, \calO)$ for some affine pscheme $(X,\calO)$ is called
		{\bf global}.
	\end{defn}
	
	\begin{prop}
		Every partial field is global.
	\end{prop}
	\begin{proof}
		If $F$ is a partial field and $\Spec(F) = (X, \calO),$ then $X = \Set{ (0) }$ is a point.
		Then it is clear that $\calO_X(X) = F.$
	\end{proof}
	 
\subsection{Partial schemes}
\begin{defn}[partial scheme]
	A locally pringed space is called a {\bf partial scheme} if it is locally an affine pscheme.
\end{defn}

\subsection{Projective Spaces}

In the polynomial semiring $\NN[y_0, \dots, y_n]$ of $n+1$ indeterminates, we can consider a partial
monoid
\begin{align*}
	\NN_1[y_0, \dots, y_n]=\Set{ f(y_0, \dots, y_n) | \mbox{~the constant term of~}f \mbox{~is~}0\mbox{~or~}1 }, 
\end{align*}
in which two polynomials are summable if constant terms of them are summable in $\fun.$
Let $B$ be a partial subring of $\NN_1[x_0, \dots, x_n]$ which contains $\angles{y_0, \dots, y_n},$ the 
multiplicatively written free commutative monoid thought of as a partial ring with trivial sum.
Consider a multiplicative subset $S_i = \Set{ y_i^r | r\in \NN},$ then the localization $S_i^{-1} B$
has a $\ZZ$ grading by the degree of polynomials. Let
$A^{(i)}$ denote the 0-th part of $S_i^{-1} B\,(0\leq i \leq n).$ More precisely,
\begin{align*}
	A^{(i)} &= \Set{y_i^{-r} f_r | r\in \NN, f_r\in B \mbox{~: a homogeneous polynomial of degree~}r},\\
	A^{(i)} &= \Set{ (y_i^{-r} f_r,y_i^{-r} g_r) | (f_r, g_r ) \in B_2}.
\end{align*}
If $A^{(i,j)}$ denotes the localization of $A^{(i)}$ by the multiplicative subset \[\Set{(y_j/y_i)^r | r\in \NN},\]
then we have $A^{(i,j)}=A^{(j,i)}.$ This observation ensures that, for any partial ring $A,$ we can glue
affine pschemes $\AAA_i = \Spec A^{(i)}$ along the isomorphisms of open subschemes 
$\Spec A^{(i,j)}$ along the isomorphisms given by the equalities $A^{(i,j)}=A^{(j,i)}.$
Let $\PP^n_V$ denote the resulting partial scheme, where $V = \Spec(B),$ which is thought of as
a vector space determined by $B.$

\begin{example}\label{ex:projective}
	Let 
	\[
		B = \fun \angles {y_0, \dots, y_n | \exists(y_0 + \dots + y_n) }.
	\]
	At this point, we propose to think of this partial ring as the most fundamental one
	which is between $\angles{y_0, \dots, y_n}$ and $\NN_1[x_0, \dots, x_n].$ For example
	$\Hom_{\alg_{\fun}}(B, A)$ equals $A_{n+1},$ the summable $(n+1)$-tuples in $A,$
	which may be think of as the most fundamental $\fun$-module ``of rank $n+1$''.
	In this example, $\PP^n_V $ for this specific $B$ is abbreviated as $\PP^n.$
	
	In this case,
	\begin{align*}
		A^{(i)} &= \Set{\mbox{~subsum of~}(y_0/y_i + \dots + y_n/y_i)^r | r\in \NN}\cup \set{0},\\
		A^{(i,j)} &= \Set{y_i^s y_j^{-s}\left(\mbox{~subsum of~}(y_0/y_i + \dots + y_n/y_i)^r\right) |r\in \NN, s\in \ZZ}\cup \set{0},\\
		A^{(i,j,k)} &= \Set{
		\begin{array}{l}y_i^s y_j^t y_k^u\times\\
		\left(\mbox{~subsum of~}(y_0/y_i + \dots + y_n/y_i)^r\right)
		\end{array} |\begin{array}{l}
		r\in \NN, \\
		s,t,u\in \ZZ,\\
		s+t+u = 0
		\end{array}
		}\cup \set{0},\\
		\vdots	
	\end{align*}
	and so on. 
	
	Now let $F$ be a partial field. Then by Corollary \ref{cor:one-one-correspondence},
	there is a one to one correspondence between the set of $F$ points $\Spec F \to \AAA_i$ and
	\[
		A_i(F) = \Hom_{\alg_{\fun}}(A^{(i)}, F).
	\]
	Then the set of $F$-points $\Spec F\to \PP^n$ can be identified with
	\[
		\PP^n(F) = \coprod_{i=1}^n A_i(F)/\sim,
	\]
	where for $v_i \in A_i(F)$ and $v_j \in A_j(F),$ $v_i\sim v_j$
	if there exists a homomorphism $v\colon A^{(i,j)} \to F$ such that $v|_{A^{(i)}} = v_i$ and $v|_{A^{(j)}} = v_j.$
	
	For any subset $\set{i_1,\dots, i_r} \subseteq \Set{0,1,\dots , n},$ we put
	\[
		A_{i_1,\dots, i_r}(F) = \Hom_{\alg_{\fun}}\left(A^{(i_1,\dots, i_r)}, F \right).
	\]
	Then as sets,
	\[A_{i_1,\dots, i_r}(F) \cong
			\left(\begin{array}{l}
				n\mbox{-tuples~} (x_1, \dots, x_n)\mbox{~for which}\\
				1+x_1 + \dots +x_n \mbox{~can be calculated in~}F\\
				\mbox{and~} x_1, \dots , x_{r-1} \neq 0
			\end{array}\right).
	\]
	Then
	\[
		\# A_{i_1,\dots, i_r}(F)=  (\kappa-1)^{r-1}\kappa^{n-r+1},
	\]
	where $\kappa = \kappa(F)$ denotes the number of elements of $F$ which is summable with 1.
	Now we can calculate $\# \PP^n(F )$ as
	\begin{align*}
		\# \PP^n(F) &= \sum_{i=1}^{n+1} (-1)^{i-1} \binom{n+1}{k}(\kappa-1)^{i-1}\kappa^{n-i+1}\\
		&= -\frac{(\kappa-(\kappa-1))^{n+1} - \kappa^{n+1}}{\kappa-1}\\
		&= \frac{\kappa^{n+1} - 1}{\kappa-1} = \kappa^n + \dots + \kappa + 1.
	\end{align*}
	Of course, $\kappa(\FF_q) = q,$ where $\FF_q$ denotes the finite field with $q$ elements, 
	and $\kappa(\fun ) = 1.$
\end{example}

\section{Affine Group PSchemes and Affine PGroup PSchemes}
\subsection{$\GG_a$}
	Put $K = \fun\angles{ x, y | \exists (x+y) }$ and let $R$ be the congruence $\angles{(x+y, 0)}.$
	Then put $G =  K/ R.$ If we put $t = [x] = -[y] \in G,$ then
	\[
		G = \Set{ a_0 + a_1t + \dots + a_r t^r | r\in \NN, a_0 =0\mbox{~or~}1, a_i \in \ZZ }.
	\]
	A cogroup structure on $G$ is given by 
	\begin{align*}
		&m\colon G\to G\otimes G~;~ t\mapsto 1\otimes t + t\otimes 1\\
		&e\colon G\to \fun ~;~t \mapsto 0\\
		&i\colon G\to G~;~t\mapsto -t.
	\end{align*}
	Then $\GG_a = \Hom_{{\alg}_{\fun}} (G, -)$ is the additive group.
\subsection{$\GG_m$}
	Put $K= \angles{ x, y }$ and let $R$ be the congruence $\angles{(ab, 1)}.$
	Then put $G = K/R.$ If we put $t = [x] \in G, $ then $G = \Set{ t^n | t\in \ZZ}.$
	A cogroup structure on $G$ is given by 
	\begin{align*}
		&m\colon G\to G\otimes G~;~ t\mapsto t\otimes t\\
		&e\colon G\to \fun ~;~t \mapsto 1\\
		&i\colon G\to G~;~t\mapsto t^{-1}.
	\end{align*}
	Then $\GG_m = \Hom_{{\alg}_{\fun}} (G, -)$ is the multiplicative group.
\subsection{$\GGLL_n$}
In this section, we construct an affine pgroup pscheme $\GGLL_n,$
which induces a $\grp$ valued functor when restricted to a `good' partial rings.

\subsubsection{Linear Algebra}
Let $A$ be a partial ring.
\begin{defn}[$A$-modules]
	An $A$-module is a partial monoid $M$ equipped with a bilinear action $A\times M\to M$ by $A.$
\end{defn}

For any natural number $n, A^n$ is an $A$-module in a natural way. 
Let $\varphi \colon A^n \to A^m$ be an $A$-module homomorphism.
Let $\set{e_1,\dots, e_n}$ and\\ $\set{f_1,\dots, f_m}$ be the canonical basis of $A^n, $
respectively. Suppose
$\varphi( e_j ) = \sum_{i=1}^m c_{ij} f_i.$
Let $(a_1,\dots, a_n)\in A^n$ be any element. 
Since $( a_1e_1, \dots, a_n e_n )$ is a summable $n$-tuple in $A^n,$
$( a_1\varphi(e_1), \dots, a_n \varphi(e_n) )$ is a summable $n$-tuple in $A^m.$
This implies that $( a_1 c_{1j} , a_2c_{2j}, \dots, a_n c_{nj} ) $ is a summable $n$-tuple in $A.$
Now, we put
\[
	A_{(n)} = \Set{ (c_1, \dots, c_n) \in A^n | 
		\begin{array}{l}
			(a_1 c_1 , \dots , a_n c_n) \mbox{~is a summable~}n\mbox{-tuple}\\
			\mbox{for any } a_1, \dots, a_n \in A
		\end{array}	
	}.
\]
If we think of $C = (c_{ij})$ as an $m\times n$ matrix, then we have shown above that 
each row of $C$ is an element of $A_{(n)}.$ 
The set of $m\times n$ matrices with this property is denoted by $M_{m,n}(A).$
Conversely, if we are given an $m\times n$
matrix $C = (c_{ij})\in M_{m,n}(A),$ we can define a $A$-module homomorphism 
$\varphi \colon A^n\to A^m$
by the usual matrix multiplication $\varphi(a_1,\dots, a_n) = C\,^t\!(a_1,\dots, a_n).$

If $m=n,$ matrix multiplication gives $M_n(A) = M_{n,n}(A)$ a non-commutative monoid structure.
The invertible elements of $M_n(A)$ constitute a group, which is denoted by $GL_n(A).$

We introduce a weaker version $M'_{m,n}(A)$ of $M_{m,n}(A),$ in which matrices have their rows in $A_n,$
instead of $A_{(n)}.$
If $m=n,$ matrix multiplication gives $M'_n(A) = M'_{n,n}(A)$, in this case, 
a non-commutative partial magma structure as is illustrated below.
If we put
\begin{align*}
	M &= M'_n(A), \mbox{~and}\\
	M_2 &= \Set{ ( (c_{ij}), (d_{ij}) ) | (c_{i1}d_{1j}, \dots, c_{in}d_{nj}) \in A_n \mbox{~for each~} i, j },
\end{align*}
then a multiplication $M_2 \to M$ is given by the matrix multiplication. The unit matrix
gives a unit for the partial magma $M,$ which can be multiplied to any element of $M.$
As a definition of invertible matrix in $M'_n(A)$, we give the following one (hinted by the group pscheme 
construction given in a following section): 
a matrix $C \in M'_n(A)$ is invertible if there exists a matrix $C'\in M'_n(A)$ such that we can form
$CC'$ and $C'C$ in $M'_n(A)$ and $CC' = C'C = I_n,$ the unit matrix.
The invertible elements of $M_n(A)$ constitute a partial group, which is denoted by $GL'_n(A),$
with the following (ad hoc) definition of a partial group.

\begin{defn}[partial group]
	A partial group is a non-commutative 
	partial magma $G$ in which for every element $a\in G,$ there exists
	an element $b\in G$ such that $ab$ and $ba$ can be formed in $G$ and $ab=ba=1,$ where
	$1$ is the unit element of the partial magma $G.$	
\end{defn}

\begin{defn}[good partial ring]
	A partial ring $A$ is called good if $A_n = A_{(n)}.$ 
\end{defn}

If $A$ is a good partial ring, $GL_n(A) = GL'_n(A).$
Note that commutative monoids and commutative rings are good partial rings.

\subsubsection{Partial Cogroups}

If $\calC$ is a category with finite coproducts, then we can define a partial cogroup in $\calC.$
Let $I$ be an initial object of $\calC,$ and $\otimes $ be a binary coproduct in $\calC.$

\begin{defn}[Partial cogroup]
	A partial cogroup in $\calC$ is
	\begin{enumerate}[label = \arabic*)]
		\item an object $G,$
		\item an object $H$ and an epimorphism $j\colon G\otimes G\to H,$
		\item a morphism $e\colon G\to I,$ called the counit, and epimorphisms
		$e_L\colon H \to I\otimes G$ and $e_R \colon H \to G\otimes I,$
		\item a morphism $m\colon G\to H,$ called the comultiplication, and
		\item a morphism $i\colon G\to G,$ called the inverse, and morphisms
		$i_L, i_R \colon H \to G$
	\end{enumerate}
	which makes the following diagrams commute :
		\begin{enumerate}[label = \alph*)]
		\item 
		\[
			\begin{tikzcd}
				& G\otimes G \ar{ld}[above]{e\otimes id}\ar[two heads]{d}{j}\ar{rd}{id\otimes e}\\
				I\otimes G\ar{d}[left]{(id,!)} & H \ar[two heads]{l}{e_L}\ar[two heads]{r}[below]{e_R} & G\otimes I\ar{d}{(!, id)}\\
				G & G \ar[equal]{l}\ar{u}{m}\ar[equal]{r} & G
			\end{tikzcd}
		\]
		\item 
		\[
			\begin{tikzcd}
				& G\otimes G\ar{ld}[above]{i\otimes id}\ar{rd}{id\otimes i}\ar[two heads]{d}{j}\\
				G & H\ar{l}{i_L}\ar{r}[below]{i_R}  & G\\
				I\ar{u}{!} & G\ar{u}{m}\ar{l}{e}\ar{r}[below]{e} & I\ar{u}{!}.
			\end{tikzcd}
		\]
	\end{enumerate}
\end{defn}

\subsubsection{$\GGLL_n$}
	Let $\NN[x_{ij}, y_{ij}\,(1\leq i,j \leq n) ]$ be the semiring of polynomials of
	$2n^2$ indeterminates $x_{ij}, y_{ij}\,(1\leq i,j \leq n).$
	Consider $n\times n$ matrices $X = (x_{ij}), Y = (y_{ij}), Z = XY = (z_{ij})$ and $W = YX = (w_{ij}).$
	Let $K$ be the subset of $\NN[x_{ij}, y_{ij}\,(1\leq i,j \leq n) ]$ 
	consisting of $4n$ elements
	\[
		\begin{array}{l}
				x_i = x_{i1}+\dots +x_{in} \,(1\leq i\leq n),\\
				y_i = y_{i1}+\dots +y_{in} \,(1\leq i\leq n),\\		
				z_i = z_{i1}+\dots +z_{in} \,(1\leq i\leq n)\mbox{~and~}\\		
				w_i = w_{i1}+\dots +w_{in} \,(1\leq i\leq n).
		\end{array}
	\]
	We put $G' = \fun \angles{x_{ij}, y_{ij}\,(1\leq i,j \leq n) | \exists t, \forall t\in K}.$
	Let $Q$ be the smallest congruence on $G'$ which contains $2n^2$ pairs
	$(z_{ij} ,\delta_{ij})$ and $(w_{ij}, \delta_{ij})\,(1\leq i, j\leq n).$
	Then we put $G=G'/Q.$
	
	Next, let $N = \NN[x_{ij}, y_{ij}, x'_{ij}, y'_{ij}\,(1\leq i,j \leq n) ]$ be the semiring of polynomials of
	$4n^2$ indeterminates $x_{ij}, y_{ij}, x'_{ij}, y'_{ij}\,(1\leq i,j \leq n).$
	Consider $n\times n$ matrices 
	\begin{align*}
		&X = (x_{ij}), Y = (y_{ij}), Z = XY = (z_{ij}), W = YX = (w_{ij}), \\
		&X' = (x'_{ij}), Y' = (y'_{ij}), Z' = X'Y' = (z'_{ij}), W' = Y'X' = (w'_{ij}), \\
		&S = XX' = (s_{ij}), T = Y'Y = (t_{ij}), U = ST = (u_{ij}), V = TS = (v_{ij}).
	\end{align*}
	We put
	\[
		L = \Set{ x_i, y_i, z_i, w_i, x'_i, y'_i, z'_i, w'_i, s_i, t_i, u_i, v_i | 1\leq i \leq n},
	\]
	where $*_i$ denotes the sum of $i$-th row of a matrix indicated by the capital of the same letter $*.$
	We put $H' = \fun \angles{x_{ij}, y_{ij},x'_{ij}, y'_{ij}\,(1\leq i,j \leq n) | \exists t, \forall t\in L}.$
	Let $R$ be the smallest congruence on $H'$ which contains $6n^2$ pairs
	$(z_{ij} ,\delta_{ij}),$ $(w_{ij}, \delta_{ij}),$ $(z'_{ij} ,\delta_{ij}),$  $(w'_{ij}, \delta_{ij}),$
	$(u_{ij} ,\delta_{ij})$ and $(v_{ij}, \delta_{ij})\,(1\leq i, j\leq n).$ Then we put $H = H'/R.$
	
	The following list of maps define partial ring homomorphisms which give
	$G$ a partial cogroup structure:
	\begin{align*}
		& j\colon G\otimes G \to H~&;~ & j(x_{ij}\otimes 1) = x_{ij}, j(y_{ij}\otimes 1) = y_{ij},\\
		& & &  j(1\otimes x_{ij}) = x'_{ij}, j(1\otimes y_{ij}) = y'_{ij},\\
		& e\colon G \to \fun~&;~& e(x_{ij}) = e(y_{ij}) = \delta_{ij},\\
		& e_L\colon H \to \fun\otimes G~&;~& e_L(x_{ij}) = e_L(y_{ij}) = \delta_{ij}\otimes 1,\\
		& &  & e_L(x'_{ij}) = 1\otimes x_{ij}, e_L(y'_{ij}) = 1\otimes y_{ij},\\
		& e_R\colon H \to \fun\otimes G~&;~& e_R(x_{ij}) = x_{ij}\otimes 1, e_R(y_{ij}) = y_{ij}\otimes 1,\\
		& &  & e_R(x'_{ij}) = e_R(y'_{ij}) = 1\otimes \delta_{ij},\\
		& m\colon G\to H~&;~& m(x_{ij}) = s_{ij}, m(y_{ij}) = t_{ij},\\
		& i\colon G\to G~&;~& i(x_{ij}) = y_{ij}, i(y_{ij}) = x_{ij},\\
		& i_L\colon H\to G~&;~& i_L(x_{ij}) = y_{ij}, e_L(y_{ij}) = x_{ij},\\
		& &  & e_L(x'_{ij}) = x_{ij}, e_L(y'_{ij}) = y_{ij},\\
		& i_R\colon H\to G~&;~&  i_R(x_{ij}) = x_{ij}, e_R(y_{ij}) = y_{ij},\\
		& &  & e_R(x'_{ij}) = y_{ij}, e_R(y'_{ij}) = x_{ij}.
	\end{align*}
	
	\begin{thm}
		There exists a representable functor
		$\GGLL_n \colon \pring \to \pgrp$
		from the category of partial rings to the category of partial groups
		which enjoys the following properties:
		\begin{enumerate}
			\item its restriction to the category of good partial rings factors through 
			$\grp, $ the category of groups.
			\item $\GGLL_n (A)$ is 
			the group of $n$-th general linear group with entries in $A$, if
			$A$ is a  commutative rings with 1, and
			\item $\GGLL_n (\fun) = \frakS_n$ is $n$-th symmetric group.
		\end{enumerate}	 
	\end{thm}

	\bibliographystyle{plain}
	\bibliography{reference}
\end{document}